\documentclass[reqno,12pt]{amsart}

\usepackage{amsmath}
\usepackage{amsfonts}
\usepackage{amssymb}
\usepackage{eufrak}
\usepackage{amscd}
\usepackage{amsthm}
\usepackage{amstext}
\usepackage[all]{xy}
\usepackage{pb-diagram}

   \theoremstyle{plain}
   \newtheorem{thm}{Theorem}[section]
   \newtheorem{prop}[thm]{Proposition}
   
   \newtheorem{cor}[thm]{Corollary}
   \theoremstyle{definition}
   
   \newtheorem{defn}[thm]{Definition}
   \newtheorem{example}[thm]{Example}
   \theoremstyle{remark}

\author{A. Korchagin}

\date{}

\address{Moscow Institute of Physics and Technology, Dolgoprudny, 141701, Russia}

\email{mogilevmedved@yandex.ru}

\thanks{The author acknowledges partial support by the RFBR grant No. 18-01-00398.}

\title{On diagonal actions of free group on the Cantor set}

\begin{document}
\maketitle

\begin{abstract}
We study diagonal actions $\varphi:\mathbb{F}_2\curvearrowright\partial\mathbb{F}_2\times K$ on the Cantor set which are given by $\varphi_a=\partial_a\times\alpha,\varphi_b=\partial_b\times\beta$. Under some restrictions on $\alpha,\beta$ we compute $K_*(C(\partial\mathbb{F}_2\times K)\rtimes_r\mathbb{F}_2)$. As an application in the case of $\alpha$ is Denjoy homeomorphism of the Cantor set and $\beta=id$ we will show that $C(\partial\mathbb{F}_2\times K)\rtimes_r\mathbb{F}_2$ is Kirchberg algebra with $K_*(C(\partial\mathbb{F}_2\times K)\rtimes_r\mathbb{F}_2)=(\mathbb{Z}^\infty,\mathbb{Z}^\infty)$. Also we will check that $C^*$-crossed product by Denjoy homeomorphism on the Cantor set is $C^*$-algebra generated by weighted shift, namely $C(K)\rtimes\mathbb{Z}\cong C^*(T_x)$ where $x\in \{1,2\}^\mathbb{Z}$ is two-sided Fibonacci sequence.
\end{abstract}

\section*{Introduction}

It was proven in \cite{Ph} the following fundamental result: let $A,B$ be unital Kirchberg algebras (i.e. separable, simple, nuclear, purely infinite $C^*$-algebras) in UCT-class. Then $A\cong B$ iff $K_*(A)\cong K_*(B)$ and $[1]_A=[1]_B$; moreover for all $(G_0,G_1,g)$ where $G_i$ are countable commutative groups and $g\in G_0$ there exits unital Kirchberg algebra in UCT class such that $(K_0(A),K_1(A),[1]_A)=(G_0,G_1,g)$. While we have complete classification it is useful to have the concrete realization for all Kirchberg algebras. There are a lot of results in this direction, for example it is known that all Kirchberg algebras in UCT-class with finitely generated $K_0$-group and free $K_1$-group can be realized as graph algebras (see \cite{SpSemi}). This result
allows to prove that these algebras are semiprojective. There exist more general models for Kirchberg algebras: for example all Kirchberg algebras can be realized as $C^*$-algebra of k-graph (see \cite{SpG}) or $C^*$-algebra of topological graph (see \cite{Kat}). But this realizations are rather difficult to work with it. 

Another approach to study Kirchberg algebras is to realize their as crossed product of the form $C(X)\rtimes_r G$ since there are a lot of tools of working with crossed products. One of the first result in this direction is construction two non-isomorphic Kirchberg algebras which both of them can be realized as $C(K)\rtimes_r\mathbb{F}_2$, where $K$ is the Cantor set (see \cite{ES}). The article \cite{Suz} was the serious breakthrough. Suzuki constructed uncountable family of pairwise non-isomorphic Kirchberg algebras which can be realized as $C(\partial\mathbb{F}_n\times K)\rtimes_r\mathbb{F}_n$ where crossed product is taken by {\it diagonal action} $\varphi:\mathbb{F}_n\curvearrowright \partial\mathbb{F}_n\times K$ which is defined via formula $\varphi_{g_j}=\partial_{g_j}\times \alpha_j$ where $\alpha_j$ are some homeomorphisms of $K$ (here $K$ is the Cantor set) and $\partial:\mathbb{F}_n\curvearrowright\partial\mathbb{F}_n$ is left side multiplication action on Gromov boundary. The benefit of diagonal action is its automatic topological freeness and amenability. Suzuki consider the case when $\alpha_j$ have nice finite dimensional approximation and compute $K_*$-groups for corresponding Kircberg algebras.

In this article we compute $K_*(C(\partial\mathbb{F}_2\times K)\rtimes_r\mathbb{F}_2)$ for some diagonal actions $\mathbb{F}_2\curvearrowright
\partial\mathbb{F}_2\times K$ which are not covered by results of \cite{Suz}. Our model example would be Denjoy diagonal action $\varphi:\mathbb{F}_2\curvearrowright\partial\mathbb{F}_2\times K$ which is defined via $\varphi_a=\partial_a\times\alpha$, $\varphi_b=\partial_b\times id$. For corresponding crossed product we have $K_*(C(\partial\mathbb{F}_2\times K)\rtimes_r\mathbb{F}_2)=(\mathbb{Z}^\infty,\mathbb{Z}^\infty)$. Remark that $K_*$-groups of Suzuki examples have torsion. All these results are attempts to understand which Kirchberg algebras in UCT-class can be realized as $C(X)\rtimes G$ for some $X$ and $G$.

The second section is about concrete realization of algebra $C(K)\rtimes \mathbb{Z}$ where $K$ is the Cantor set and crossed product is taken by Denjoy homeomorphism. We merge some well known results from theory of $C^*$-algebras and dynamic theory and get that such algebra can be realized as $C^*$-algebra generated by weighted shift.

The author is grateful to V.M. Manuilov for helpful discussions and useful comments.

\section{Computation of $K$-theory of crossed product}\label{Section1}

We will need the following proposition:

\begin{prop}\label{L1}{\rm (\cite{ES}, Proposition 2.1)}
Let $G$ be a countable group, $K$ be the Cantor set and $\alpha:G\curvearrowright K$ be an action. Then $C^*$-algebra $C(K)\rtimes_r G$ is Kirchberg algebra precisely when the following conditions hold:

\begin{enumerate}
\item
$\alpha$ is minimal,
\item
$\alpha$ is topological free,
\item
$\alpha$ is amenable,
\item
Every non-zero projection $p\in C(K)$ is infinite in $C(K)\rtimes_r G$.
\end{enumerate}
\end{prop}

It is well known that $\partial\mathbb{F}_2$ with topology induced from the product topology is homeomorphic to the Cantor set. Let $\partial:\mathbb{F}_2\curvearrowright\partial\mathbb{F}_2$ be the boundary action which is minimal, topological free and amenable (c.f. \cite{O} and \cite{Suz}).

Let us consider $X=K\times\partial\mathbb{F}_2$ and diagonal action $\varphi:\mathbb{F}_2\curvearrowright X$ given by $\varphi_a=\alpha\times\partial_a$, $\varphi_b=\beta\times\partial_b$ for some homeomorphisms $\alpha$ and $\beta$ of the Cantor set $K$. Also denote by $\alpha$, $\beta$ induced actions on $C(K)$, $K_0(C(K))=C(K,\mathbb{Z})$ and also denote by $\alpha$, $\beta$ actions on $K_0(C(K))^n$ given by $\alpha(f_1,\ldots,f_n)=(\alpha f_1,\ldots ,\alpha f_n)$ (this will not become ambiguity later).

For every finite reduced word $\omega$ over alphabet $\{a,a^{-1},b,b^{-1}\}$ let us consider characteristic function $p_\omega\in C(\mathbb{F}_2,\mathbb{Z})$ of the clopen set $\{g\in\partial\mathbb{F}_2|g=\omega g^\prime\text{ - infinite reduced word}\}$. Let us denote by $\omega_*$ first letter of word $\omega$.

\begin{prop}
Let $\varphi:\mathbb{F}_2\curvearrowright X$ be as above. Let $\alpha\beta\alpha^{-1}\beta^{-1}=1$ and $\alpha$ is minimal. Then $\phi$ satisfies conditions of Proposition 1.1 and so $C(X)\rtimes_r\mathbb{F}_2$ is Kirchberg algebra (cross product is taken by $\varphi$).
\end{prop}

\begin{proof}
Since $\partial$ is topological free and amenable, it follows that $\varphi$ is also topological free and amenable (preserving these properties by direct product is one of the motivations of studying diagonal actions).

Let us prove minimality of $\varphi$, i.e. every orbit is dense. Fix $(x,\omega_1),(y,\omega_2)\in K\times\mathbb{F}_2=X$. It is enough to check that for every $\varepsilon$ there exists $g\in\mathbb{F}_2$ such that
$$\rho_X(g\cdot (x,\omega_1),(y,\omega_2))=\rho_K(g\cdot x,y)+\rho_{\partial\mathbb{F}_2}(\omega_1,\omega_2)<\varepsilon$$
Let $m=[\log_2(1/\varepsilon)]+2$ and
$$g=\omega_2[m]\sigma_1ba^Nb\sigma_2$$
where $\omega_2[m]$ is $m$-symbol prefix of infinite word $\omega_2$, $\sigma_1\in\{a,a^{-1},b,b^{-1}\}$ is arbitrary letter such that there is no reduction in word $\omega_2[m]\sigma_1b$ and $\sigma_2$ is a letter for which there is no reduction in $b\sigma_2\omega_1$. $N$ will be defined later. It is easy to check that $\rho_{\mathbb{F}_2}(g\cdot \omega_1,\omega_2)<\varepsilon/2$.

Let $N_1$ (resp. $N_2$) be the sum of all powers of $a$ (resp. of $b$) in $\omega_2[m]\sigma_1b^2\sigma_2$. Since $\alpha$ and $\beta$ commute, we have $g\cdot x=\alpha^N(\alpha^{N_1}\beta^{N_2}(x)).$ Since $\alpha$ is minimal we can find $N$ such that $\rho_K(\alpha^N(\alpha^{N_1}\beta^{N_2}(x)),y)<\varepsilon/2$. So the action $\varphi$ is also minimal.

Fix some clopen base $\Omega$ for topology of $K$, and consider clopen base $\{supp(p_\omega)\}$ for topology of $\partial\mathbb{F}_2$ where $\omega$ runs over all finite words. For every projection $p\in C(K\times\partial\mathbb{F}_2)$ we can find finite sum decomposition:
$$p=\sum_jp_{U_j}\otimes p_{\omega_j}$$
for some $\omega_j$ and $U_j\in\Omega$ (here $p_U$ is characteristic function of $U$). It is easy to check that sum of two orthogonal infinite projection is also infinite projection.  So it is enough to prove that $p_U\otimes p_\omega$ is infinite in $C(X)\rtimes_r\mathbb{F}_2$. Put
$$g=\omega\sigma_1ba^Nb^Ma\sigma_2$$
where $\sigma_1$ is a letter from alphabet $\{a,a^{-1},b,b^{-1}\}$ for which there is no reduction in word $\omega\sigma_1b$ and $\sigma_2$ is a letter for which there is no reduction in $a\sigma_2\omega$. Choose $N$ and $M$ such that sum of all powers of $a$ and sum of all powers of $b$ equal to zero. We have $p_{g\cdot\omega}\subset p_\omega$ and $p_{g\cdot x}\neq p_\omega$. Since $\alpha$ and $\beta$ commute, we obtain that $g\cdot U=U$ (it is easy to check that $g|_K=id$). So we have strict embedding:
$$g\cdot(p_U\otimes p_\omega)=p_U\otimes p_{g\cdot\omega}\subset p_U\otimes p_\omega$$
So every projection in $C(K)$ is infinite in $C(K)\rtimes_r\mathbb{F}_2$.
\end{proof}

Let us remind that there exists sex-terms exact sequence for K-theory of reduced crossed product of arbitrary $C^*$-algebra by free group (see \cite{PimVoi}, Theorem 3.5):
$$
\begin{diagram}
\node{K_0(A)^n} \arrow{e,t}{\zeta}
\node{K_0(A)} \arrow{e}
\node{K_0(A\rtimes_r\mathbb{F}_n)} \arrow{s}\\
\node{K_1(A\rtimes_r\mathbb{F}_n)}\arrow{n}
\node{K_1(A)} \arrow{w}
\node{K_1(A)^n}\arrow{w,t}{\zeta}
\end{diagram}
$$
Here $\zeta(\gamma_1,\gamma_2,\ldots,\gamma_n)=(1-a_1^*)\gamma_1+\ldots+(1-a_n^*)\gamma_n$, where $a_1,\ldots a_n$ are free generators of $\mathbb{F}_n$, $a_j^*$ are induced actions on $K_*(A)$. In case of $A=C(K)$ (where $K$ is Cantor set) the above exact sequence becomes the following exact sequence:
$$0\to K_1(C(K)\rtimes_r\mathbb{F}_2)\to C(K,\mathbb{Z})^2\xrightarrow{\zeta}C(K,\mathbb{Z})\to K_0(C(K)\rtimes_r\mathbb{F}_2)\to 0$$
So we have
$$K_0(C(K)\rtimes_r\mathbb{F}_2)=C(K,\mathbb{Z})\slash\{a_j^*(f)=f\}$$
$$K_1(C(K)\rtimes_r\mathbb{F}_2)=Ker(\zeta)$$
Here quotient in $K_0$ is taken by subgroup generated by all $a_j^*(f)-f$ where $j\in\{1,2\}$ and $f\in C(K,\mathbb{Z})$.

\begin{prop}
Let $\varphi:\mathbb{F}_2\curvearrowright X$ be as above. Then for corresponding crossed product we have:
$$K_0(C(X)\rtimes_r\mathbb{F}_2)=
C(K,\mathbb{Z})^2\slash
\scriptsize\begin{array}{c}\\ \\
((2-\alpha-\alpha^{-1})f,(\alpha^{-1}-1)(\beta-1)f)=0\\ ((\beta^{-1}-1)(\alpha-1)f,(2-\beta-\beta^{-1})f)=0
\end{array}\normalsize
$$
\end{prop}

\begin{proof}
Consider arbitrary $f\in C(K,\mathbb{Z})$. Put $f_a=(f,0,0,0), f_{a^{-1}}=(0,f,0,0), f_b=(0,0,f,0), f_{b^{-1}}=(0,0,0,f),\Delta_af=f_a+f_{a^{-1}},\Delta_bf=f_b+f_{b^{-1}}$. From six-terms exact sequence we can get isomorphism:
$$K_0(C(X)\rtimes_r\mathbb{F}_2)=C(K\times\partial\mathbb{F}_2,\mathbb{Z})\slash
\scriptsize\begin{array}{c}\\ \\
\alpha\times\partial_a(F)=F\\
\beta\times\partial_b(F)=F
\end{array}\normalsize
$$
Here $F$ runs over all $C(K\times\partial\mathbb{F}_2)$. First let us construct isomorphism:
$$C(K\times\partial\mathbb{F}_2,\mathbb{Z})\slash
\scriptsize\begin{array}{c}\\ \\
\alpha\times\partial_a(F)=F\\
\beta\times\partial_b(F)=F
\end{array}\normalsize
\overset{\Phi}{\underset{\Psi}{\rightleftharpoons}}
C(K,\mathbb{Z})^4\slash
\scriptsize\begin{array}{c}\\ \\ \\ \\
\alpha f_a=f_a+\Delta_b f\\
\alpha^{-1}f_{a^{-1}}=f_{a^{-1}}+\Delta_b f\\
\beta f_b=f_b+\Delta_a f\\
\beta^{-1}f_{b^{-1}}=f_{b^{-1}}+\Delta_a f
\end{array}\normalsize
$$
Define $\Psi$ on the generators in the following way:
$$\Psi(f_x)=f\otimes p_x,\:\:\:\text{где}\:\: \: x\in\{a,a^{-1},b,b^{-1}\}$$
It is easy to check that this map is homomorphism since "right relations" is a special case of "left relations".

To construct $\Phi$ notice first that functions of the form $f\otimes p_\omega$ generate $C(K\times\partial\mathbb{F}_2)$ where $f\in C(K)$ and $\omega$ is finite word. Put $\omega=x_1x_2\ldots x_n$ where $x_j\in \{a,a^{-1},b,b^{-1}\}$. Then we have:
$$[f\otimes p_\omega]=[x_{n-1}^{-1}\ldots x_2^{-1}x_1^{-1}(f\otimes p_\omega)]=[f^\prime\otimes p_{x_n}]$$
where $f^\prime=x_{n-1}^{-1}\ldots x_2^{-1} x_1^{-1}(f)$ (in case of single letter word multiplication has some special feature. You should keep in mind that $\partial_{a^{-1}}p_a=p_a+p_b+p_{b^{-1}}$). So in any equivalence class $[f\otimes p_\omega]$ we can find element of the form $f^\prime\otimes p_{\omega^\prime}$ such that $\omega^\prime$ is single letter. So we can construct map
$$C(K\times\partial\mathbb{F}_2,\mathbb{Z})\slash\scriptsize\begin{array}{c}\\ \\
\alpha\times\partial_a(F)=F\\
\beta\times\partial_b(F)=F
\end{array}\normalsize\to C(K,\mathbb{Z})^4$$
$$[f\otimes p_\omega]\mapsto f^\prime_{\omega^\prime}$$
It is easy to see that quotient map does not depend on choice of $f^\prime$ and $\omega^\prime$ and it is well defined homomorphism. Let us prove it more careful. Consider arbitrary $g=\sum g_j\otimes p_{\omega_j}\in C(K\times\partial\mathbb{F}_2,\mathbb{Z})$ and put
$$\Phi(g)=\sum(((\omega_j)_*\omega_j^{-1})(g_j))_{(\omega_j)_*}$$
I think it it better to comment this formula: here we should calculate $(\omega_j)_*\omega_j^{-1}$ word in alphabet $\{\alpha,\alpha^{-1},\beta,\beta^{-1}\}$ instead of $\{a,a^{-1},b,b^{-1}\}$, then apply resulting homomorphism to $g_j$ and finally we should take element in $C(K,\mathbb{Z})^4$ whose three coordinates equal to zero and $\omega_j)_*$-$th$ coordinate is equal to result of previous step. For example $\Phi(f\otimes p_{babba})=(\beta^{-2}\alpha^{-1}\beta^{-1}(f))_a$.

It is easy to see from definition that $\Phi$ is homomorphism (since it preserves sums). Let us check that it is well defined, i.e. result does not depend on realization of $g$ as sum. First consider the case when $g=f\otimes p_\omega$ (without loss of generality we can suppose that $\omega_*=a$). Then $f\otimes p_a=f\otimes p_{\omega a}+f\otimes p_{\omega b}+f\otimes p_{\omega b^{-1}}$. Then we have
$$\Phi(f\otimes p_{\omega})=(\alpha\omega^{-1}(f))_a=$$
$$=(\omega^{-1}(f))_a+(\omega^{-1}(f))_b+(\omega^{-1}(f))_{b^{-1}})=\Phi(t\otimes p_{\omega a})+\Phi(f\otimes p_{\omega b})+\Phi(f\otimes p_{\omega b^{-1}})$$
The first and third equation is just relation $\alpha f_a=f_a+\Delta_b f$ for $\omega^{-1}(f)$. So the result does not change under refinement.

Now let we have for $g\in C(K\times\partial\mathbb{F}_2)$ two decompositions: $g=\sum g_j\otimes p_{\omega_j}$ and $g=\sum g^\prime_j\otimes p_{\omega^\prime_j}$. Then we can consider small enough refinement $\{supp(\tilde\omega_j)\}$ for topology of $\partial\mathbb{F}_2$ such that equation $\sum g_j\otimes p_{\tilde\omega_j}=\sum g^\prime_j\otimes p_{\tilde\omega_j}$ holds just in the case of $g_j=g_j^\prime$ for all $j$. Since we have checked that $\Phi$ does not change under refinement we can conclude that $\Phi$ is well defined.

Let us prove that $\Phi$ preserve relations (without loss of generality we can consider first relation $\alpha\times\partial_a(F)=F$). Since $\Phi$ preserve sums we only need to check it for $F=f\otimes p_\omega$. Consider first case $\omega=a^{-1}$. Since $\partial_a p_{a^{-1}}=p_{a^{-1}}+p_b+p_{b^{-1}}$ we have
$$\Phi(t\otimes p_{a^{-1}}-\alpha\times\partial_a(f\otimes p_{a^{-1}}))=\Phi(f\otimes p_{a^{-1}}-\alpha f\otimes (p_a+p_b+p_{b^{-1}}))=$$
$$=f_{a^{-1}}-(\alpha f)_{a^{-1}}-(\alpha f)_b-(\alpha f)_{b^{-1}}$$
So $\Phi$ maps relation to relation (since right-hand side is just $\alpha f_a-f_a-\Delta_b f=0$ for function $\alpha f$). Consider the second case $\omega\neq a^{-1}$. Since $\omega_*=(a\omega)_*$ and $\partial_a p_\omega=p_{a\omega}$ we can conclude
$$\Phi(f\otimes p_\omega)-\alpha\times\partial_a(f\otimes p_\omega))=(\omega_*\omega_{-1}f)_{\omega_*}-\Phi(\alpha f\otimes p_{a\omega})=$$
$$=(\omega_*\omega_{-1}f)_{\omega_*}-((a\omega)_*\omega^{-1}\alpha^{-1}\alpha f)_{(a\omega)_*}=0$$
So $\Phi$ and $\Psi$ are well defined homomorphisms and it is easy to see that $\Phi\circ \Psi=id$ and $\Psi\circ\Phi=id$.

Now let us construct isomorphisms
$$C(K,\mathbb{Z})^4\slash
\scriptsize\begin{array}{c}\\ \\ \\ \\
\alpha f_a=f_a+\Delta_b f\\
\alpha^{-1}f_{a^{-1}}=f_{a^{-1}}+\Delta_b f\\
\beta f_b=f_b+\Delta_a f\\
\beta^{-1}f_{b^{-1}}=f_{b^{-1}}+\Delta_a f
\end{array}\normalsize
\overset{\Gamma}{\underset{\Theta}{\rightleftharpoons}}
C(K,\mathbb{Z})^2\slash
\scriptsize\begin{array}{c}\\ \\
((2-\alpha-\alpha^{-1})f,(\alpha^{-1}-1)(\beta-1)f)=0\\ ((\beta^{-1}-1)(\alpha-1)f,(2-\beta-\beta^{-1})f)=0
\end{array}\normalsize $$
We are going to use $\mathbb{Z}$-Gaussian elimination which will not change quotient but can simplify relations. We have
$$
\begin{cases}
\alpha f_a=f_a+\Delta_b f\\
\alpha^{-1}f_{a^{-1}}=f_{a^{-1}}+\Delta_b f\\
\beta f_b=f_b+\Delta_a f\\
\beta^{-1}f_{b^{-1}}=f_{b^{-1}}+\Delta_a f
\end{cases}
\Longleftrightarrow
\begin{cases}
\alpha f_a-f_a=\Delta_b f\\
\alpha^{-1}f_{a^{-1}}-f_{a^{-1}}=\Delta_b f\\
\beta f_b-f_b=\Delta_a f\\
\beta^{-1}f_{b^{-1}}-f_{b^{-1}}=\Delta_a f
\end{cases}
\Longleftrightarrow
$$
$$
\Longleftrightarrow
\begin{cases}
\alpha f_a-f_a=\Delta_b f\\
\alpha^{-1}f_{a^{-1}}-f_{a^{-1}}=\alpha f_a-f_a\\
\beta f_b-f_b=\Delta_a f\\
\beta^{-1}f_{b^{-1}}-f_{b^{-1}}=\beta f_b-f_b
\end{cases}
\Longleftrightarrow
\begin{cases}
\alpha f_a-f_a=\Delta_b f\\
(\alpha^{-1}-1)(\Delta_a f-f_a)=\alpha f_a-f_a\\
\beta f_b-f_b=\Delta_a f\\
(\beta^{-1}-1)(\Delta_b-f_b)=\beta f_b-f_b
\end{cases}
\Longrightarrow
$$
$$
\Longrightarrow
\begin{cases}
\alpha f_a-f_a=\Delta_b f\\
(\alpha^{-1}-1)((\beta-1)f_b-f_a)=\alpha f_a-f_a\\
\beta f_b-f_b=\Delta_a f\\
(\beta^{-1}-1)((\alpha-1)f_a-f_b)=\beta f_b-f_b
\end{cases}
\Longrightarrow
\begin{cases}
\alpha f_a-f_a=\Delta_b f\\
(\alpha^{-1}-1)(\beta-1)f_b+(2-\alpha-\alpha^{-1})f_a=0\\
\beta f_b-f_b=\Delta_a f\\
(\beta^{-1}-1)(\alpha-1)f_a+(2-\beta-\beta^{-1})f_b=0
\end{cases}
$$
Since sets $\{f_a,f_{a^{-1}},f_b,f_{b^{-1}}\}$ and $\{f_a,\Delta_af,f_b,\Delta_bf\}$ generate each other and since first and third final relation can be read as definition of $\Delta_a f$ and $\Delta_b f$ we can reduce generator set to $\{f_a,f_b\}$ with second and forth final equations as relations.

Now let us define $\Gamma$ on new generators in the following way (remind that $f_{a^{-1}}=\Delta_a f-f_a$ and $f_{b^{-1}}=\Delta_b f-f_b$):
$$\Gamma(f_a)=(f,0)$$
$$\Gamma(f_b)=(0,f)$$
$$\Gamma(f_{a^{-1}})=(-f,(\beta-1)f)$$
$$\Gamma(f_{b^{-1}})=((\alpha-1)f,-f)$$
Define $\Theta$ on the generators in the following way:
$$\Theta((f,0))=f_a$$
$$\Theta((0,f))=f_b$$
It is easy to see that they are homomorphisms and $\Theta\circ\Gamma=id$, $\Gamma\circ\Theta=id$.
\end{proof}

\begin{prop}
Let $\varphi:\mathbb{F}_2\curvearrowright X$ be as above, $\alpha\beta\alpha^{-1}\beta^{-1}=1$ and $Im(\alpha-1)$ is infinitely generated group. Then
$$K_1(C(X)\rtimes_r\mathbb{F}_2)=\mathbb{Z}^\infty$$
\end{prop}

\begin{proof}
It is well known that $C(X,\mathbb{Z})=\mathbb{Z}^\infty$ (see \cite{PSS} for example). From six-termed exact sequence it is easy to deduce that $K_1(C(X)\rtimes_r\mathbb{F}_2)$ is subgroup of free abelian group and so it is also free abelian group. Let us prove that it has infinitely generated.

From six-termed exact sequence it follows that
$$K_1(C(X)\rtimes_r\mathbb{F}_2)=Ker(\zeta)=\{(F,G)\in C(X,\mathbb{Z})^2:(a-1)F+(b-1)G=0\}$$
Since $\alpha\beta\alpha^{-1}\beta^{-1}=1$ we have
$$M=\{(1\otimes (1-\beta)f,1\otimes (\alpha-1)f)\}\subset Ker(\zeta)$$
Since $Im(\alpha-1)$ is infinitely generated we conclude that $M$ is infinitely generated and hence $Ker(\zeta)$ is infinitely generated too.
\end{proof}

\begin{cor}
Let $\varphi:\mathbb{F}_2\curvearrowright X$ be as above, $\beta=1$ and $\alpha$ be minimal. Then $C(X)\rtimes_r\mathbb{F}_2$ is Kirchberg algebra and moreover
$$K_0(C(X)\rtimes_r\mathbb{F}_2)=\mathbb{Z}^\infty\oplus \left(C(K,\mathbb{Z})\slash
\small\begin{array}{c}\\
\{(\alpha-1)^2f=0\}
\end{array}\normalsize\right)$$
$$K_1(C(X)\rtimes_r\mathbb{F}_2)=\mathbb{Z}^\infty$$
\end{cor}

\begin{proof}
This is easy follows from Propositions 1.1-1.4. Since $\alpha$ is minimal from Proposition 1.2 we conclude that $C(X)\rtimes_r\mathbb{F}_2$ is Kirchberg algebra.

Next from Proposition 1.3 we get
$$K_0(C(X)\rtimes_r\mathbb{F}_2)=C(K,\mathbb{Z})^2\slash\{((2-\alpha-\alpha^{-1})f,0)=0\}$$
Since $\alpha$ is automorphism then it is surjection, so $\{\alpha f\}_{f\in C(K,\mathbb{F})}=C(K,\mathbb{F})$ and we can deduce that relation $((2-\alpha-\alpha^{-1})f,0)=0$ equivalent to $(2-\alpha-\alpha^{-1})\alpha f,0)=0$. Since there are no any relations on the second coordinate to conclude the calculation of $K_0$ it is enough to notice that $C(K,\mathbb{Z})=\mathbb{Z}^\infty$.

Next for calculation of $K_1$ according to Proposition 1.4 it is enough to check that $Im(\alpha-1)$ is infinitely generated, i.e. for every $n$ we can not generate $Im(\alpha-1)$ with $n$ elements. That's mean that for every $n$ there is surjection $Im(\alpha-1)\to \mathbb{Z}^n$. To prove this let us consider arbitrary $x\in K$. Since $\alpha$ is minimal all $x,\alpha(x),\ldots ,\alpha^n(x)$ are different points. Consider some clopen disjoint neighborhoods $V_j$ of $x_j$. Also consider clopen sets $U=V_0\cap \alpha^{-1}(V_1)\cap\ldots\cap \alpha^{-n}(V_n)$ and $U_j=\alpha^j(U)$. It is easy to see that $U_j$ are disjoint neighborhoods of $\alpha^j(x)$ (since $U_j\subset V_j$). Consider subgroup $M=\left\{\sum\limits_{j=1}^{n}y_j(\chi_{U_j}-\chi_{U_{j-1}}):y_j\in\mathbb{Z} \right\}\subset Im(\alpha-1)$.
It is easy to see that there exists epimorphism $M\to\mathbb{Z}^n$ which is defined by formula $f\mapsto (f(x),f(\alpha(x)),\ldots,f(\alpha^{n-1}(x)))$. So $M$ can not be generated by $n-1$ elements. So $Im(\alpha-1)$ is infinitely generated and so $K_1(C(X)\rtimes_r\mathbb{F}_2)=\mathbb{Z}^\infty$.
\end{proof}

Let us recall the construction of Denjoy homeomorphism of Cantor set (more information you can find for example in \cite{PSS}). Let $T$ be the circle, $\lambda$ be $\pi$-irrational number such that $\lambda<\pi$. Let $\theta_\lambda$ is rotation through an angle $\lambda$. Consider arbitrary
point $a_0\in T$ and its orbit $O=\{a_j\}_{j\in\mathbb{Z}}$ under rotation (here $a_j=\theta_\lambda^j(a_0)$. This orbit is dense and $K=T\backslash O$ is invariant under $\theta_\lambda$ Next let us cut circle $T$ at $a_j$ and put instead it intervals $I_j$ of length $|I_j|=2^{-|j|}$. To be more precisely consider metric function $m$ on $K$ which is defined by
formula $m(x,y)=|[x,y]^*|+\sum_j2^{-|j|}$, where sum runs over such $j$ that $a_j\in[x,y]^*$; and $[x,y]^*$ is arc of the circle from $x$ to $y$ clockwise. Let $|[x,y]^*|$ be the length of it. Let us define distance between $x,y$ by the formula $\rho(x,y)=\min\{m(x,y),m(y,x)\}$ (we should notice that $m$ is not symmetric). It is easy to check that $m$ is metric which induces the topology on $K$. It is easy to see that $K$ is totally disconnected in this topology and have no isolate point. So $K$ is the Cantor set. As clopen base for its topology we can take $\{[a_i,a_j]\}_{i,j\in\mathbb{Z}}$ where $[a_i,a_j]=K\cap [a_i,a_j]^*$. We will call $\theta_\lambda|_K$ Denjoy homeomorphism of the Cantor set with rotation number $\lambda$.

\begin{example}
Let $\beta=1$, $\alpha$ be Denjoy homeomorphism of the Cantor set with rotation number $\lambda$ and $\varphi:\mathbb{F}_2\curvearrowright X$ be as above. Then $C(X)\rtimes_r\mathbb{F}_2$ is Kirchberg algebra and
$$K_0(C(X)\rtimes_r\mathbb{F}_2)=\mathbb{Z}^\infty$$
$$K_1(C(X)\rtimes_r\mathbb{F}_2)=\mathbb{Z}^\infty$$
\end{example}

Notice that in Example 1.6 the action $\varphi$ can not be represent as projective limit of actions on finite sets. So this example does not cover by results of \cite{Suz} where in every cases $K_0$-group has a torsion.

\begin{proof}
It is well known that Denjoy homeomorphism of the Cantor set is minimal and so it is enough to compute $C(K,\mathbb{Z})\slash\{(\alpha-1)^2f=0\}$. Consider $x_j=p_{[a_j,a_{j+1}]}$, $j\in\mathbb{Z}$ where $p_{[a,b]}$ is characteristic function of $[a,b]$. It is easy to see that $x_j$ together with $1=p_{[0,2\pi]}$ generate $C(K,\mathbb{Z})$. Let us prove that they are independent. Let $f=\mu_{-n}x_{-n}+\mu_{-n+1}x_{-n+1}+\ldots+\mu_n x_n+\mu \cdot 1=0$. Consider some two points $\iota_j$ and $\kappa_j$ respectively from right and left side from $a_j$, which are closed enough to $a_j$ to make $[\iota_l,\kappa_j]$ contain just $a_j$ among all $a_i$ with $|i|\leq n$. Then we have $f(\kappa_j)-f(\iota_j)=\mu_j-\mu_{j-1}$. So all $\mu_j$ are equal to each other. To prove that $\mu=\mu_0=0$ we consider Lebesgue measure $\nu$ on circle and measure $\nu_*=\nu|_K$ on $K\subset T$ (notice that $K$ has full measure in $T$). It is easy to see that $\nu_*([0,2\pi])=2\pi$ and $\nu_*([a_j,a_{j+1}])=\lambda$ Then we have
$$0=\int\limits_{K}f d\nu_*=2\pi\cdot\mu+(2n+1)\lambda\cdot\mu_0$$
Since $\lambda$ is $\pi$-irrational we deduce that $\mu=\mu_0=0$. So the linear combination $f$ is trivial and thus $C(K,\mathbb{Z})=\langle 1,x_j\rangle_{j\in\mathbb{Z}}=\langle 1\rangle\oplus\langle x_j\rangle_{j\in\mathbb{Z}}=\mathbb{Z}\oplus\mathbb{Z}^\infty$.
By the construction of Denjoy homeomorphism  we have $\alpha(1)=1$ and $\alpha(x_j)=x_{j+1}$ and so
$$C(K,\mathbb{Z})\slash\{(\alpha-1)^2f=0\}=\mathbb{Z}\oplus\left(\mathbb{Z}^\infty\slash
\small\begin{array}{c}\\
\{x_{j+2}=2x_{j+1}-x_j\}
\end{array}\normalsize\right)
=\mathbb{Z}\oplus\mathbb{Z}^2$$
The last isomorphism is defined on the second summand by the formula $x_0\mapsto 0\oplus (1,0)$, $x_1\mapsto 0\oplus (0,1)$. It is well defined on all other generators since in quotient using relations $x_{j+2}=2x_{j+1}-x_j$ and $x_{-j-2}=2x_{j-1}-x_j$ we can represent all $x_j$ as linear combination of $x_0$ and $x_1$.
So we get
$$K_0(C(X)\rtimes_r\mathbb{F}_2)=\mathbb{Z}^\infty$$
From the Lemma 1.4 we can easy deduce that $K_1(C(X)\rtimes_r\mathbb{F}_2)=\mathbb{Z}^\infty$.
\end{proof}

\section{Operator realization of Denjoy homeomorphism}\label{Section2}

In this section we are going to represent $C(K)\rtimes\mathbb{Z}$ (where crossed product is taken by Denjoy homeomorphism) as $C^*$-algebra generated by weighted shift. Let us remind that {\it joint spectrum} of commuting family of operators $\{B_i\}_{i\in\Omega}$ is set of such $\lambda\in\mathbb{C}^\Omega$ that for every finite subset $\Gamma\subset\Omega$ operator $B_\Gamma=\sum\limits_{i\in\Gamma}(B_i-\Lambda_i)(B_i-\Lambda_i)^*+(B_i-\lambda_i)^*(B_i-\lambda_i)$ is not invertible. Let us denote joint spectrum by $\sigma(\{B_i\})$. Joint spectrum is compact subset of $\mathbb{C}^\Omega$ with respect to product topology. Let $C^*(x_1,x_2,\ldots)$ be $C^*$-algebra which is generated by operators $1,x_1,x_2,\ldots$. The following theorem holds:

\begin{thm}
(\cite{POW}, Theorem 1.4) Let $\{B_j\}_{j\in\Omega}$ be commuting family of normal operators. Then $C^*(\{B_j\})=C(\sigma(\{B_j\}))$.
\end{thm}

Consider some sequence of bounded sequences $a_i=\{a_i(j)\}\in\mathbb{C}^\Omega$. Let $X_i\in B(\ell^2(\mathbb{Z}))$ be diagonal operator which is defined via formula $X_i(e_j)=a_i(j)e_j$. Consider also symmetric sequences $b_i(j)=a_j(i)$.

\begin{prop}
The following equation holds: $\sigma(\{X_i\})=\overline{\{b_i\}_{i\in\mathbb{Z}}}\subset \mathbb{C}^\mathbb{Z}$
\end{prop}

\begin{proof}
The proof is easy exercise. By $\overline{\{b_i\}_{i\in\mathbb{Z}}}$ we mean closure.
\end{proof}

So we have $C^*(\{X_i\})=C\left(\overline{\{b_i\}_{i\in\mathbb{Z}}}\right)$, i.e. it is possible to describe spectrum of $C^*$-algebra $C^*(\{X_j\})$ rather concrete in terms of sequence $b_i$ (aka $a_i$). It is natura to ask whether exist algebraic conditions or invariants of sequence $a_i$ which know topological information about $\overline{\{b_i\}_{i\in\mathbb{Z}}}$, for example is it possible to find algebraic conditions which can guarantee that this spectrum homeomorphic to sphere $S^n$? But we will not discuss this topic in this article.

Let $x\in\mathbb{C}^\mathbb{Z}$ be bounded sequence and let $T_x$ be corresponding weighted shift which is defined by formula $T_xe_i=x_ie_{i+1}$. It is natural to ask is it possible classify $C^*$-algebra $C^*(T_x)$ in terms of combinatoric or algebraic properties of sequence $x$. It is not known much about such algebras. Next we will need usual shift operator $T$ which can be defined by formula $Te_i=e_{i+1}$. Also let us consider sequences $a_i$ which are $i$-shifts of $x$ and corresponding diagonal operators $X_i$, which are defined via $X_ie_j=a_i(j)e_j=x_{i+j}e_j$. It is easy to see that $T^{-1}X_iT=X_{i+1}$.

\begin{prop}
Let $x\in[1,2]^\mathbb{Z}$. Then $C^*(T_x)=C^*(T,\{X_i\}_{i\in\mathbb{Z}})$.
\end{prop}

We should remark that $[1,2]$ is not essential here. This proposition also holds in the case when $x$ is positive, bounded and bounded away from zero.

\begin{proof}
On the one hand $T_x=TX_0$, so you can generate $T_x$ by $T$ and $X_0$. On the other hand since $T_x^*e_i=x_{i-1}e_i$ and $x$ is positive and bounded away from zero we get that $X_0$ is invertible and also $X_0=\sqrt{T_x^*T_x}$, $T_x=TX_0^{-1}$, $X_n=T^{-n}X_0T^n$. 
\end{proof}

\begin{prop}
Let $x\in[1,2]^\mathbb{Z}$. Then we have $C^*(T_x)=C(K)\rtimes\mathbb{Z}$ where $C(X)=C^*(\ldots,X_{-1},X_0,X_1,\ldots)$ and the action is defined via $\alpha(X_i)=X_{i-1}=TX_iT^{-1}$
\end{prop}

\begin{proof}
By proposition 2.3 we have $C^*(T_x)=C^*(T,\{X_i\}_{i\in\mathbb{Z}})$. Let $\pi: C(T,\{X_i\}_{i\in\mathbb{Z}}\subset B(\ell^2(\mathbb{Z}))$. Next we are going to consider homomorphism $\pi_\infty=\oplus_{j=-\infty}^{+\infty}T^{-j}\pi T^{j}:C^*(T,\{X_i\})\hookrightarrow B(\oplus_{-\infty}^{+\infty}\ell^2(\mathbb{Z}))$ which is defined via formula (here $|$ is separator for between coordinates with positive and negative indexes)
$$\pi_\infty(a)=\ldots T \pi(a)T^{-1}\oplus |\pi(a)\oplus T^{-1}\pi(a) T\oplus T^{-2}\pi(a) T^{2}\oplus \ldots$$
It is easy to see that we have $C^*$-algebra isomorphism $Im(\pi)\cong Im(\pi_\infty)$. Since $\mathbb{Z}$ is amenable group we can conclude that $C(X)\rtimes_r\mathbb{Z}=C(X)\rtimes\mathbb{Z}$. Remind that reduced crossed product defined as completion of image of $*$-algebraic homomorphism $\gamma: C(X)\rtimes_{alg}\mathbb{Z}\to B(H\otimes\ell^2(\mathbb{Z}))$ which is defined via $\gamma(f)(\xi\otimes e_j)=\omega(\alpha^{-j}(f))\xi\otimes e_j$, $\gamma(u)(\xi\otimes e_j)=\xi\otimes e_{j+1}$ where $f\in C(X)$ and $u$ is generator of $\mathbb{Z}$; here $\alpha:\mathbb{Z}\curvearrowright X$ is action and $\omega :C(X)\hookrightarrow B(H)$ is arbitrary embedding (the resulting crossed product does not depend on embedding. More information about construction and properties of $C^*$-algebra crossed products reader can find in Section 4.1 of \cite{BO}). In our case let $\omega :C(X)\subset B(\ell^2(\mathbb{Z}))$ be identical inclusion (i.e. $\omega=\pi|_{C(X)}$) and consider isomorphism $\ell^2(\mathbb{Z})\otimes\ell^2(\mathbb{Z})=\oplus_{-\infty}^{+\infty}\ell^2(\mathbb{Z})$ we get that our $\gamma$ acts via formulas:
$$\gamma(f)(\ldots,\xi_{-1},|\xi_0,\xi_1,\xi_2,\ldots)=(\ldots,\alpha(f)\xi_{-1},|f\xi,\alpha^{-1}(f)\xi_1,\alpha^{-2}(f)\xi_2,\ldots)$$
$$\gamma(u)(\ldots,\xi_{-1},|\xi_0,\xi_1,\xi_2,\ldots)=(\ldots,\xi_{0},|\xi_1,\xi_2,\xi_3,\ldots)$$
To prove Proposition 2.4 it is enough to check that $\gamma$ and $\pi_\infty$ are unitary conjugate. Let $\{[e_i]_j\}_{i\in\mathbb{Z}}$ be standard basis for $j$-component of $\oplus_{-\infty}^{+\infty}\ell^2(\mathbb{Z})$. It is easy to see that we can define conjugation unitary $U$ via formula $U[e_i]_j=[e_j]_i$. We have
$$\pi_\infty(X_0)[e_i]_j=x_{i+j}[e_i]_j=U^{-1}x_{i+j}[e_j]_i=U^{-1}\gamma(X_0)[e_j]_i=U^{-1}\gamma(X_0)U[e_i]_j$$
$$\pi_\infty(T)[e_i]_j=[e_{i+1}]_j=U^{-1}[e_j]_{i+1}=U^{-1}\gamma(T)[e_j]_i=U^{-1}\gamma(T)U[e_i]_j$$
Since operators $X_0$ and $T$ generate $C^*(T,\{X_i\})$ we get that $\pi_\infty=U^{-1}\gamma U$.
\end{proof}

\begin{prop}
Let $x$ be a periodic sequence with period $n$. Then $A_x= M_n\otimes C(S^1)$.
\end{prop}

\begin{proof}
Almost the same fact is proved in Proposition V.3.1 of \cite{Dav}. The proof can be transfer here almost without any changes.
\end{proof}

It is well know the following fact but we give a proof since it is not long:

\begin{prop}
Let $\theta$ be irrational number, $x_n=2\cos(2\pi n\theta)+3)$. Then $A_x=A_\theta$ where $A_\theta=\langle u,v| uv=e^{2\pi i\theta}vu; u,v\text{ - unitary}\rangle$ is irrational rotation $C^*$-algebra.
\end{prop}

\begin{proof}
Since $x\in[1,5]^\mathbb{Z}$ we can use Proposition 2.3. Let $T$ be standard two side shift and $D$ be diagonal operator which is defined via $De_n=e^{2\pi i n\theta} e_n$. Since $D$ and $T$ satisfy universal universal relation of $A_\theta$ and since this algebra is simple we get $A_\theta=C^*(T,D)$. So we only need to generate $T_x$ by $T$ and $D$ and conversely.

On the one hand we have $T_x=(D+D^*+3)T$.

On the other hand using $T_x=(D+D^++3)T$ it is not difficult to deduce that $D+D^*=\sqrt{T_x^*T_x}-3$ and $T=(\sqrt{T_x^*T_x})^{-1}T_x$. Using universal relation we can get $T(D+D^*)T^*=e^{2\pi i \theta }D+e^{-2\pi i \theta}D^*$. Now it is not difficult to represent $D$ as linear combination of $D+D^*$ and $e^{2\pi i \theta }D+e^{-2\pi i \theta}D^*$.
\end{proof}

\begin{prop}
Let $x\in\mathbb{C}^\mathbb{Z}$ such that $|x_n|=1$. Then $A_x=C(S^1)$.
\end{prop}

\begin{proof}
Using classification of commutative $C^*$-algebras we only need to check $\sigma(T_x)=S^1$ where $\sigma(T_x)$ is spectrum of $T_x$ and $S^1$ is unit circle. Since $T_x$ unitary we have $\sigma(T_x)\subset S^1$. On the other hand spectrum contain such $\lambda\in\mathbb{C}$ for which there exists sequence of vectors $\xi_n$ such that $\|\xi_n\|=1$ and $\|(T-\lambda)\xi_n\|\to 0$.  Consider arbitrary $\lambda$ with $|\lambda|=1$ and we can take $\xi_n=\chi_n/\|\chi_n\|$ where $\chi_n=\sum\limits_{j=0}^{n}\alpha_je_j$ where $\alpha_0=1$ and $\alpha_{j+1}=\alpha_jx_{j-1}/\lambda$. So for every $\lambda$ with $|\lambda|=1$ we construct desired sequence and thus $\sigma(T_x)=S^1$.
\end{proof}

\begin{prop}
Let $x\in\{1,2\}^\mathbb{Z}$ be such sequence that we can find in it segments of 1 or 2 of arbitrary long length. Then $A_x$ is not simple.
\end{prop}

\begin{proof}
By Proposition 2.4 we have $A_x=C(X)\rtimes\mathbb{Z}$ where $X=\overline{\{b_i\}_{i\in\mathbb{Z}}}\subset \mathbb{C}^\mathbb{Z}$ and the action is just left shift (recall that $b_i(j)=x_{i+j}$). It is easy to see that the sequence $\bf{2}$ belongs to $X$ since we have the segments of 2 with arbitrary long length (respectively $\bf{1}$ in case of segments of 1), here $\bf{2}$ is sequence such that ${\bf{2}}_i=2$ for every $i$. Since $\bf{2}$ is invariant we get closed subspace invariant under shift and so $A_x$ can not be simple.
\end{proof}

\begin{example}
Let $x\in\{1,2\}^\mathbb{Z}$ be a sequence in which we can find any finite word in alphabet $\{1,2\}$ then $A_x$ is not simple.
\end{example}

It is easy corollary of Proposition 2.8. In this case it is not difficult to see that $X=\{1,2\}^\mathbb{Z}$ (since every $y\in\{1,2\}^\mathbb{Z}$ belongs to closure of orbit of $x$). In general case $X$ is union of some closed invariant under shift subspaces of $\{1,2\}^\mathbb{Z}$. In this case $X$ is union of all closed invariant subspaces and so in some sense $A_x$ is the maximum non simple $C^*$-algebra in this case.

By the Proposition 2.4 in case of $x$ bounded, positive and bounded away from zero all $A_x$ have the form $C(X)\rtimes\mathbb{Z}$ and so rather interesting to apply dynamic theory here. The dynamic systems on the Cantor set are very important in dynamic systems theory since on the one hand they have very nice finite-dimensional approximation and also the Cantor set has tons of interpretation, so we have a lot of instruments to work with then. And on the other hand Cantor systems have universal property, i.e. for every separable compact dynamic system $\mathbb{Z}\curvearrowright X$ there exists dynamic system on the Cantor set $\mathbb{Z}\curvearrowright K$ which admits equivariant surjection $K\to X$, that is dynamic analog of well known topological fact that every separable compact metric space $X$ admits surjection map $K\to X$ of the Cantor set. Minimal Cantor dynamic admits full classification in the following sense:

\begin{thm} (\cite{GPS}, Theorem 2.1) Let $\alpha,\beta\in Homeo(K)$ some minimal homeomorphisms of the Cantor set. Then $C(K)\rtimes_\alpha\mathbb{Z}\cong C(K)\rtimes_\beta\mathbb{Z}$ iff when there exists positive isomorphism $K_0(C(K)\rtimes_\alpha\mathbb{Z})\cong K_0(C(K)\rtimes_\beta\mathbb{Z})$
\end{thm}

Positivity mean that isomorphism maps positive part of $K_0$-group to positive part. We should remark that in \cite{GPS} more general result is proved: $C(K)\rtimes_\alpha\mathbb{Z}\cong C(K)\rtimes_\beta\mathbb{Z}$ iff $\mathbb{Z}\curvearrowright_\alpha K$ and $\mathbb{Z}\curvearrowright_\beta K$ are strong orbit equivalent but we do not need it.

Let $x\in\{1,2\}^\mathbb{Z}$. Then {\it block growth} is number $p_n(x)$ of different subwords of $x$ of length $n$.

\begin{defn}
The sequence $x\in\{1,2\}^\mathbb{Z}$ is called {\it Sturmian sequence} if $p_n(x)=n+1$.
\end{defn}

It is easy to see that for Sturmian sequence $x$ there exists limit $\alpha(x)=\lim\limits_{N\to\infty}\frac{|P_N(x)|}{2N+1}$ which is called {\it slope}, here $P_N(x)$ is the number of $2$ in the $[-N,N]$-segment (for more information see for example Prop. 6.1.10 in \cite{Fo}). Also we have that $\alpha$ is irrational number (otherwise $x$ would be periodic, but periodic sequence can not be Sturmian). 

Let $\Sigma_\alpha=\overline{\{T^n(\omega)\}_n}\subset\{1,2\}^\mathbb{Z}$ be closure of all shifts of Sturnian sequence $\omega$ with slope $\alpha$. It is known that $\Sigma_\alpha$ is homeomorphic to the Cantor set and does not depend of choice of $\omega$: it is depend just on slope (we recommend to see \cite{SMM} for more information).

Let us remind that {\it Fibonacci sequence} is $\tilde x\in\{1,2\}^\mathbb{N}$ which is defined in the following way: $f_0=1$, $f_1=21$, $f_{n+1}=f_nf_{n-1}$, $\tilde x=\lim f_n$, i.e. $\tilde x=1212112112121\ldots$ It is the most famous example of Sturmian sequence. We need to modify this sequence a bit to get sequence in $\{1,2\}^\mathbb{Z}$.

First let us notice that $f_{n+3}=f_{n+1}f_nf_{n+1}$ so if we consider subsequence $f_{3n}$ it growths not just in right but in all two side. So we get sequence in $\{1,2\}^\mathbb{Z}$, more precisely:
\begin{center}
$$\:|f_0$$
$$\:\: f_1|f_0f_1$$
$$f_4f_1|f_0f_4$$
$$\:\:\:f_7f_4f_1|f_0f_1f_4f_7$$
\end{center}
where $|$ is separator between coordinates with positive and negative index. Remark that on $n$-th step we get $f_{3n}$ which is shifted to the left. Let $x\in\{1,2\}^\mathbb{Z}$ be the limit of these words. It is clear that $x$ has the same subwords as $\tilde x$, so $x$ also Sturmian. It is known that $\alpha(x)=\frac{1}{\phi^2}$ where $\phi$ is golden ratio.

\begin{prop}
Let $x\in\{1,2\}^\mathbb{Z}$ be $\mathbb{Z}$-Fibonacci sequence as above, $T_x$ be corresponding weighted shift, $A_x=C^*(T_x)$. Then
$$C(K)\rtimes\mathbb{Z}\cong A_x$$
where crossed product is taken by Denjoy homeomorphism on the the the the Cantor set with rotation number $\frac{2\pi}{\varphi^2}$
\end{prop}

\begin{proof}
It is known that $\mathbb{Z}\curvearrowright \Sigma_\alpha$ conjugates to Denjoy system $\mathbb{Z}\curvearrowright K$ with rotation number $2\pi\alpha$ for all irrational numbers $\alpha$, i.e. $C(\Sigma_\alpha)\rtimes\mathbb{Z}\cong C(K)\rtimes\mathbb{Z}$ (see for example Theorem 2.11 from \cite{SMM}). Let us recall that $X_i\in B(\ell^2(\mathbb{Z}))$ are diagonal operators which are defined via $X_ie_j=x_{i+j}e_j$ and $\Sigma_\alpha=\overline{\{T^n(\omega)\}_n}$. By the Proposition 2.2 we have $C(\Sigma_\alpha)=C(\{X_i\})$, moreover natural isomorphism is $\mathbb{Z}$-equivariant, where $\mathbb{Z}$ acts on $C(\{X_i\})$ via conjugation by $T$. From the universal property we get epimorphism
$$C(\Sigma_\alpha)\rtimes\mathbb{Z}\to C^*(\{X_i\}, T)$$
Since action $\mathbb{Z}\curvearrowright K$ is minimal and free we have that $C^*$-algebra $C(\Sigma_\alpha)\rtimes\mathbb{Z}$ is simple, so $C(\Sigma_\alpha)\rtimes\mathbb{Z}\cong C^*(\{X_i\}, T)$. By Proposition 2.3 we have
$$C(K)\rtimes\mathbb{Z}\cong A_x$$
So finally we get $A_x\cong C^*(T_x)$.
\end{proof}

{\bf Open problems}

It is not much known about $C^*$-algebras generated by weighted shift. We would like to end with some questions which seems to be very interesting.

1) By Proposition 2.4 for $x\in [1,2]^\mathbb{Z}$ we have $C(T_x)=C(X)\rtimes\mathbb{Z}$. It would be nice to compute $K_0(C(T_x))$ (or get some exact sequence for it) in terms of combinatorial or algebraic properties of $x$. Six-terms exact sequence for crossed product by $\mathbb{Z}$ seems to be powerful tool, but it is also not clear how combinatorial properties of $x$ can recover topological information about $X$: this also seems to be very interesting question.

2) Find some classifiable class of $C^*$-algebras generated by weighted shift, i.e. find some $\Omega\subset [1,2]^\mathbb{Z}$ and some algebraic combinatorial invariant $Alg(\cdot)$ such that for every $x,y\in\Omega$ we would have:
$$C^*(T_x)\cong C^*(T_y)\:\:\:\Longleftrightarrow\:\:\: Alg(x)=Alg(y)$$
Let us remark that Proposition 2.12 together with Theorem 2.10 of \cite{PSS} can be iterpretated in the following way: for set $\Omega=\{x\in\{1,2\}^\mathbb{Z}: p_n(x)=n+1\}$ of Sturmian sequences and for every $x,y\in \Omega$ we have
$$C^*(T_x)\cong C^*(T_y)\:\:\Longleftrightarrow\:\: \alpha(x)=\alpha(y)\:\:\text{or}\:\:\alpha(x)=1-\alpha(y)$$
It would be interesting to extend $\Omega$ or find some another. It is unknown, for example, is it possible to classify in algebraic combinatorial terms $C^*$-algebras generated by weighted shift with weights from
$$\Omega=\{x\in\{1,2\}^\mathbb{Z}:\text{ for all } n \text{ holds } n+1\leq p_n(x)\leq n+2\}$$

%


\end{document}